\documentclass[a4paper,12pt]{article}
\usepackage{cmap}                        
\usepackage[cp1251]{inputenc}            
\usepackage[english]{babel}
\usepackage[left=2cm,right=2cm,top=2cm,bottom=2cm]{geometry} 
\usepackage[ruled,vlined]{algorithm2e}
\usepackage{amssymb}
\usepackage{amsmath, amsthm}
\theoremstyle{plain}
\newtheorem{thm}{Theorem}
\newtheorem{lem}{Lemma}

\newtheorem{cor}{Corollary}

\theoremstyle{definition}

\newtheorem{defn}{Definition}

\begin{document}

\begin{center}\Large
\textbf{Formations of Finite Groups in Polynomial Time:\\
the $\mathfrak{F}$-Radical}\normalsize

\smallskip
Viachaslau I. Murashka

 \{mvimath@yandex.ru\}

Faculty of Mathematics and Technologies of Programming,
Francisk Skorina Gomel State University, Sovetskaya 104, Gomel,
246028, Belarus\end{center}

 \begin{abstract}
  For a Baer-local (composition) Fitting formation   $\mathfrak{F}$ the polynomial time algorithm for the computation of the $\mathfrak{F}$-radical of a permutation group  is suggested. In particular it is showed how one can compute the $\mathfrak{F}$-radical in case when $\mathfrak{F}$ is a primitive saturated formation of soluble groups.  Moreover, the polynomial time algorithms for the computation of different lengthes associated with a group are presented.
 \end{abstract} 
 

\textbf{AMS}(2010). 20D10, 20B40.

\section{Introduction and the Main Result}

All groups considered here are finite. Fitting showed that the product of two normal nilpotent subgroups is again nilpotent, i.e. in every group there exists the greatest normal nilpotent subgroup which is called the Fitting subgroup.
Recall that a class of groups is a collection $\mathfrak{X}$ of groups with the property that
if $G\in\mathfrak{X} $ and if $H \simeq  G$, then $H \in \mathfrak{X}$. The greatest  normal $\mathfrak{X}$-subgroup of a group $G$ is called the $\mathfrak{X}$-radical and is denoted by $G_\mathfrak{X}$. It always exists in case when $\mathfrak{X}$ is a Fitting class.

The algorithm for computing the $\mathfrak{X}$-radical   (of a soluble group) was presented in \cite{HOFLING2001}. Note that as was mentioned in \cite{HOFLING2001} suggested there algorithm  (even when  $\mathfrak{X}$ is the class of all nilpotent groups)   for a permutation group of degree $3n$   may require to check for nilpotency $2^n$ subgroups. The main idea of that algorithm was to extended the $\mathfrak{X}$-radical from the given member of chief series to the next one. In this paper we suggest an algorithm for the computation of the $\mathfrak{X}$-radical (which runs in polynomial time for permutation groups) using a different approach.

Recall that the Fitting subgroup $\mathrm{F}(G)$, the $p$-nilpotent radical $\mathrm{O}_{p',p}(G)$, the soluble radical $R(G)$ of a group $G$ are just the intersection of centralizers of all, all divisible by $p$ and all non-abelian chief factors respectively. Also note that the generalized Fitting subgroup (the quasinilpotent radical) $\mathrm{F}^*(G)$ is the intersection of   innerizers of all chief factors of $G$. The first three radicals are associated with local Fitting formations and the last one is associated with a Baer-local (composition) Fitting formation. Shemetkov \cite{Shemetkov1988} obtained similar characterizations of the $\mathfrak{F}$-radical for a Baer-local Fitting formation $\mathfrak{F}$ using the generalization of the centralizer of chief factor. We can not use the results from \cite{Shemetkov1988} directly for two reasons. Formations $\mathfrak{F}$ in \cite{Shemetkov1988} are defined with the inner Baer-local function, i.e. $f(H/K)\subseteq \mathfrak{F}$ for all chief factors $H/K$. In some applications (such as primitive saturated formations) the functions will not necessary satisfy this condition. Also the constructive description of the generalized centralizer was not presented in \cite{Shemetkov1988}.

Let $f$ be a function which assigns to every simple group $J$ a possibly empty formation $f(J)$. Now extend the domain of $f$. If $G$ is the direct product of simple groups isomorphic to $J$, then we say that $G$ has type $J$ and let $f(G)=f(J)$. If $J$ is a cyclic group of order $p$, then let $f(p)=f(J)$. Such functions $f$ are called  Baer functions.  A formation $\mathfrak{F}$ is called Baer-local (or composition, see \cite{Shemetkov1988} or \cite[p. 4]{Guo2015}) if for some Baer function $f$
$$\mathfrak{F}=(G\mid G/C_G(H/K)\in f(H/K)\textrm{ for every chief factor }H/K\textrm{ of }G).$$
The main result of this paper is

\begin{thm}\label{thm2}
  Let $\mathfrak{F}$ be a Baer-local Fitting formation defined by $f$ such that $f(J)$ is a Fitting formation for any simple group $J$.
  Assume that $(G/K)_{f(J)}$ can be computed in polynomial time for every  $K\trianglelefteq G\leq S_n$ and a simple group $J$. Then $(G/K)_\mathfrak{F}$ can be computed in polynomial time for every  $K\trianglelefteq G\leq S_n$. 
\end{thm}

Note that if $f(J)=\emptyset$, then $G_{f(J)}=\emptyset$ for any group $J$. A Baer-local formation defined by Baer function $f$  is called local if $f(J)=\cap_{p\in\pi(J)}f(p)$ for every simple group $J$.

\begin{cor}
  Let $\mathfrak{F}$ be a local Fitting formation defined by $f$ such that $f(p)$ is a Fitting formation for  all prime $p$.
  Assume that $(G/K)_{f(p)}$ can be computed in polynomial time for every  $K\trianglelefteq G\leq S_n$. Then $(G/K)_\mathfrak{F}$ can be computed in polynomial time for every  $K\trianglelefteq G\leq S_n$.
\end{cor}

\section{Preliminaries}

Recall that a \emph{formation} is a class of groups $\mathfrak{F}$ which is   closed  under taking epimorphic images (i.e. from $G\in\mathfrak{F}$ and $N\trianglelefteq G$ it follows that $G/N\in\mathfrak{F}$)  and subdirect products (i.e. from $G/N_1\in\mathfrak{F}$ and $G/N_2\in\mathfrak{F}$ it follows that $G/(N_1\cap N_2)\in\mathfrak{F}$).
A class of groups $\mathfrak{F}$ is  called a \emph{Fitting class}   if it is normally hereditary (i.e. from $N\trianglelefteq G\in\mathfrak{F}$ it follows that $N\in\mathfrak{F}$) and $N_0$-closed (i.e. from $N_1, N_2\trianglelefteq G$ and $N_1, N_2\in\mathfrak{F}$ it follows that $N_1N_2\in\mathfrak{F}$).
If $\mathfrak{F}$ is a Fitting class and a formation, then $\mathfrak{F}$ is called the \emph{Fitting formation}.

For a class of groups $\mathfrak{X}$ recall that $E\mathfrak{X}$ denotes the class of groups with a normal series whose factors are $\mathfrak{X}$-groups. If $\mathfrak{F}$ is a Fitting formation, then $E\mathfrak{F}$ is also a Fitting formation and is closed by extensions.
Here  $\mathrm{O}_\Sigma(G)$ denotes the greatest normal subgroup of $G$ all whose composition factors belong to $\Sigma$ for a collection of simple groups $\Sigma$.


We use standard computational conventions of abstract finite groups equipped
with poly\-nomial-time procedures to compute products and inverses of elements (see \cite[Chapter 2]{Seress2003}).
For both input and output, groups are specified by generators. We will consider only $G=\langle S\rangle\leq S_n$ with $|S|\leq n^2$. If necessary, Sims' algorithm \cite[Parts 4.1 and 4.2]{Seress2003} can be used to arrange that $|S|\leq n^2$. Quotient groups are specified by generators of a group and its normal subgroup.

We need the following well known basic tools in our proofs (see, for example \cite{Kantor1990a} or \cite{Seress2003}). Note that some of them are obtained mod CFSG.

\begin{thm}\label{Basic}
  Given $A, B\trianglelefteq G = \langle S\rangle\leq S_n$ with $A\leq B$, in polynomial time one can solve the following problems:

  \begin{enumerate}

\item\label{i1}
     Find $C_{G/A}(B/A)$.

  \item\label{i5} Find a chief series for $G$ containing $A$ and $B$.

\item\label{i3.5} Test if $G/A$ is simple; if it is not, find a proper normal subgroup $N/A$ of $G/A$; if it is, find the name of $G/A$. In particular, find a type of a chief factor.

\item\label{i4} Find
$\mathrm{O}_{\Sigma}(G/A)$  for a collection $\Sigma$ of simple groups.

\item\label{i7} Given $H \leq G$,  find $H\cap A$.

\item\label{i2} Find $(G/A)'$.

    \item\label{i3}  Find $|G/A|$.




  \end{enumerate}
\end{thm}

The following lemma restricts the length of a chief series of a permutation group.

\begin{lem}[\cite{Babai1986}]\label{chain}
   Given $G \leq S_n$ every chain of subgroups of $G$ has at most $2n-3$  members for $n\geq 2$.
\end{lem}


\section{Proof of Theorem \ref{thm2}}

The idea of the theorem's proof is to obtain the $\mathfrak{F}$-radical as the intersection of generalized centralizers of chief factors in the sense of the following definition:

\begin{defn}\label{cf}
For a chief factor $H/K$ of $G$ with $f(H/K)\neq \emptyset$ let $C_{G, f}(H/K)$  be defined by a) $C_G(H/K)\subseteq C_{G, f}(H/K)$ and b) $C_{G, f}(H/K)/C_G(H/K)=(G/C_G(H/K))_{f(H/K)}$.
\end{defn}

\begin{lem}\label{center}
Let $B/A$ be a chief factor of a group  $G$ with $f(B/A)\neq\emptyset$ and  $N$ be a normal subgroup of $G$ with $N\leq A$. Then
  $$C_{G/N, f}((B/N)/(A/N))=C_{G,f}(B/A)/N.$$
\end{lem}

\begin{proof}
Note that
\begin{multline*}
C_{G/N}((B/N)/(A/N))=\{gN\in G/N\mid [gN, g_iN]=[g, g_i]N\in A/N \textrm{ for all }g_iN\in B/N\}=\\
\{g\in G\mid [g, g_i]N\in A/N \textrm{ for all } g_i\in B\}/N=C_G(B/A)/N.
\end{multline*}
  Now
  $$(G/N)/C_{G/N}((B/N)/(A/N))=(G/N)/(C_{G}(B/A)/N)\simeq G/C_{G}(B/A).$$
  This isomorphism induces the isomorphism between $f(B/A)$-radicals of the left and the right hands parts. It means that if $F/C_{G}(B/A)=(G/C_{G}(B/A))_{f(B/A)}$, then $$(F/N)/(C_{G}(B/A)/N)=(F/N)/C_{G/N}((B/N)/(A/N))=((G/N)/C_{G/N}((B/N)/(A/N)))_{f(B/A)}.$$ Thus $C_{G/N, f}((B/N)/(A/N))=C_{G,f}(B/A)/N$.
\end{proof}

\begin{lem}\label{center1}
  Let $H/K$ be a chief factor of a group $G$ with $f(H/K)\neq\emptyset$. Then $G_\mathfrak{F}\leq C_{G, f}(H/K)$.
\end{lem}

\begin{proof}
  If   $H/K\not\in \mathfrak{F}$, then from $G_\mathfrak{F}K/K\in\mathfrak{F}$ it follows that  $H/K\cap G_\mathfrak{F}K/K=K/K$. Therefore $HG_\mathfrak{F}/K=(H/K)\times G_\mathfrak{F}K/K$. It means that $G_\mathfrak{F}\leq C_G(H/K)\leq C_{G, f}(H/K)$.

Assume that $H/K\in \mathfrak{F}$ and $p\in\pi(H/K)$. Since $\mathfrak{F}$ is $N_0$-closed, we see that $T/K=HG_\mathfrak{F}/K\in\mathfrak{F}$. Now $(T/K)/C_{T/K}(M/N)\in f(M/N)=f(H/K)$ for every chief factor $M/N$ of $T/K$ below $H/K$.

If $H/K$ is abelian, then  $H/K$ is a $p$-group for some prime $p$. Let $K=H_0\trianglelefteq H_1\trianglelefteq\dots\trianglelefteq H_m=H$ be a part of chief series of $T/K$.  Then $(T/K)/C_{T/K}(H_i/H_{i-1})\in f(p)=f(H/K)$. Let $C/K=\cap_{i=1}^m C_{T/K}(H_i/H_{i-1})$. Now $(C/K)/C_{T/K}(H/K)$ is a $p$-group by [A, Corollary 12.4(a)]. Since $f(p)$ is a formation,  $(T/K)/(C/K)\in f(p)$. Thus $(T/K)/C_{T/K}(H/K)\in \mathfrak{N}_p f(p)$.
 Note that $$(T/K)/C_{T/K}(H/K)\simeq (T/K)C_{G/K}(H/K)/C_{G/K}(H/K)\trianglelefteq (G/K)/C_{G/K}(H/K).$$ Since $H/K$ is a chief factor of $G/K$, we see that $\mathrm{O}_p((G/K)/C_{G/K}(H/K))\simeq 1$ by \linebreak\cite[A, Lemma~13.6]{Doerk1992}. Hence $\mathrm{O}_p((T/K)C_{G/K}(H/K)/C_{G/K}(H/K))\simeq 1$. Thus\linebreak $(T/K)/C_{T/K}(H/K)\simeq(T/K)C_{G/K}(H/K)/C_{G/K}(H/K)\in f(p)= f(H/K)$.

If $H/K$ is non-abelian, then $H/K$ is the direct product of minimal normal subgroups $H_i/K$ of $T/K$ by \cite[A, Lemma~4.14]{Doerk1992}. Since $f(p)$ is a formation, from $C_{T/K}(H/K)=\cap_{i}C_{T/K}(H_i/K)$ it follows $(T/K)/C_{T/K}(H/K)\in f(H/K)$.

Since $f(H/K)$ is $N_0$-closed, we see that
\begin{multline*}
  (T/K)/C_{T/K}(H/K)=(T/K)/(C_{T}(H/K)/K)\simeq T/C_T(H/K)\simeq \\
  TC_G(H/K)/C_G(H/K)\subseteq (G/C_G(H/K))_{f(H/K)}=C_{G, f}(H/K)/C_G(H/K).
\end{multline*}
Therefore $G_\mathfrak{F}\leq T\leq  C_{G, f}(H/K)$.
\end{proof}

\begin{thm}\label{int}
Let $1=G_0\trianglelefteq G_1\trianglelefteq\dots\trianglelefteq G_m=G$ be a chief series of  a group $G$. Then $$G_\mathfrak{F}=\left(\bigcap_{i=1, f(G_i/G_{i-1})\neq \emptyset}^m C_{G, f}(G_i/G_{i-1})\right)_{E\mathfrak{F}}.$$
\end{thm}

\begin{proof} We assume that every  intersection of  empty collection of subgroups of $G$ coincides with $G$.
Let $D=\bigcap_{i=1, f(G_i/G_{i-1})\neq \emptyset}^m C_{G, f}(G_i/G_{i-1}).$
  From Lemma \ref{center1} it follows that $G_\mathfrak{F}\subseteq D$. From the other hand $G_\mathfrak{F}\in E\mathfrak{F}$. Thus $G_\mathfrak{F}\leq D_{E\mathfrak{F}}$. Let $1=D_0\trianglelefteq D_1\trianglelefteq\dots\trianglelefteq D_l=D$ be a part of chief series of $G$ below $ D_{E\mathfrak{F}}$. Then by Jordan-H\"{o}lder Theorem there is $\pi:\{1,\dots,l\}\rightarrow\{1,\dots,m\}$ such that $D_i/D_{i-1}$ is $G$-isomorphic to $G_{\pi(i)}/G_{\pi(i)-1}$ for all $i\in\{1,\dots,l\}$. Now $C_G(G_{\pi(i)}/G_{\pi(i)-1})=C_G(D_i/D_{i-1})$ for all $i\in\{1,\dots,l\}$. Hence $C_{G, f}(G_{\pi(i)}/G_{\pi(i)-1})=C_{G, f}(D_i/D_{i-1})$ by Definition~\ref{cf}. Note that \begin{multline*}
    D_{E\mathfrak{F}}/C_{D_{E\mathfrak{F}}}(D_i/D_{i-1})
  \simeq{D_{E\mathfrak{F}}}C_{G}(D_i/D_{i-1})/C_{G}(D_i/D_{i-1})=\\
  {D_{E\mathfrak{F}}}C_{G}(G_{\pi(i)}/G_{\pi(i)-1})/C_{G}(G_{\pi(i)}/G_{\pi(i)-1})\\
  \trianglelefteq C_{G, f}(G_{\pi(i)}/G_{\pi(i)-1})/C_{G}(G_{\pi(i)}/G_{\pi(i)-1})\in f(G_{\pi(i)}/G_{\pi(i)-1})=f(D_i/D_{i-1}).
  \end{multline*}
  Since $f(D_i/D_{i-1})$ is normally hereditary we see that
   $D_{E\mathfrak{F}}/C_{D_{E\mathfrak{F}}}(D_i/D_{i-1})\in f(D_i/D_{i-1})$.
   Since $f(D_i/D_{i-1})$ is closed under taking quotients, we see that
   $$D_{E\mathfrak{F}}/C_{D_{E\mathfrak{F}}}(H/K)\in f(D_i/D_{i-1})=f(H/K)$$ for all chief factors $H/K$ of $D_{E\mathfrak{F}}$ between $D_{i-1}$ and $D_i$ for all $i\in\{1,\dots, l\}$. From Jordan-H\"{o}lder theorem it follows that $D_{E\mathfrak{F}}/C_{D_{E\mathfrak{F}}}(H/K)\in f(H/K)$ for all chief factors $H/K$ of $D_{E\mathfrak{F}}$. Therefore $D_{E\mathfrak{F}}\in \mathfrak{F}$. Thus $D_{E\mathfrak{F}}=G_\mathfrak{F}$.
\end{proof}

\begin{lem}\label{fcent}
  Let $B/A$ be a chief factor of $G$. Then $C_{G,f}(B/A)$ can be computed in a polynomial time $($in the assumptions of Theorem \ref{thm2}$)$.
\end{lem}

\begin{proof}
Note that $T/A=C_{G/A}(B/A)$ can be computed in a polynomial time by \ref{i1} of  Theorem~\ref{Basic} and $T=C_G(B/A)$. Now $G/C_G(B/A)=G/T$. By our assumption $(G/T)_{f(B/A)}$ can be computed in a polynomial time.
\end{proof}


\begin{lem}\label{rad}
 In the assumptions of Theorem \ref{thm2} we can compute $(G/K)_{E\mathfrak{F}}$  in  polynomial time.
\end{lem}

\begin{proof}
Let $H/K$ be a non-abelian chief factor of $G$ of type $J$. We claim that $H/K\in\mathfrak{F}$ iff $(H/K)_{f(J)}=H/K$. Note that $H/K$ is a direct product of groups isomorphic to $J$. If  $(H/K)_{f(J)}=H/K$, then since $f(J)$ is normally hereditary, $J\in f(J)$. Now $J\in \mathfrak{F}$ by the definition of Baer-local formation. Since $\mathfrak{F}$ is $N_0$-closed $H/K\in \mathfrak{F}$. Assume now that $H/K\in \mathfrak{F}$. Hence $J\in \mathfrak{F}$. Since  $J$ is non-abelian,   $J\simeq J/C_J(J)\in f(J)$. Therefore $H/K=(H/K)_{f(J)}$ as a direct product of $f(J)$-groups.

Let $H/K$ be an abelian chief factor. We claim that $H/K\in\mathfrak{F}$ iff $f(H/K)\neq\emptyset$. It is clear that if $f(H/K)=\emptyset$, then $H/K\not\in\mathfrak{F}$.  Assume that $f(H/K)\neq\emptyset$. Hence $1\in f(H/K)$. Now $(H/K)/C_{H/K}(U/V)\simeq 1\in f(H/K)=f(U/V)$ for any chief factor $U/V$ of $H/K$. Thus $H/K\in\mathfrak{F}$ by the definition of Baer-local formation.

Note that the two above mentioned cases are equivalent to $H/K\in\mathfrak{F}$ iff $(H/K)'_{f(H/K)}=(H/K)'$.
Therefore if $H/K$ is a chief factor of $G$ we can check if $H/K\in\mathfrak{F}$ in polynomial time by our assumption and Theorem \ref{Basic}.

  We can compute  a chief series $K=G_0\trianglelefteq G_1\trianglelefteq\dots\trianglelefteq G_m=G$  of  a group $G$ in polynomial time by \ref{i5} of Theorem \ref{Basic}. Define $F_i$ by $F_i/K=(G_i/K)_{E\mathfrak{F}}$. Note that $F_i \cap G_{i-1} =F_{i-1} $ and $F_0=K$.  Now
  $$F_i/F_{i-1}=F_i/(F_i\cap G_{i-1})=F_i G_{i-1}/G_{i-1}\trianglelefteq G_i/G_{i-1}.$$
Therefore $F_i/F_{i-1}$ is isomorphic to either 1 of $G_i/G_{i-1}$. If $G_i/G_{i-1}\not\in\mathfrak{F}$, then $F_i=F_{i-1}$. Assume that $G_i/G_{i-1}\in\mathfrak{F}$.  If $G_i/G_{i-1}$ is a  group of type $J$, then $F_i/F_{i-1}$ is also such a group. Since $E\mathfrak{F}$ is closed under extensions, we see that $F_i/F_{i-1}=\mathrm{O}_{\{J\}}(G_i/F_{i-1})$. Hence $F_i$ can be computed in polynomial time by \ref{i3.5} and \ref{i4} of Theorem \ref{Basic}.
 \end{proof}

\begin{algorithm}[H]
\caption{EFRADICAL$(G, K, \mathfrak{F})$}
\SetAlgoLined
\KwResult{$(G/K)_{E\mathfrak{F}}$.}
\KwData{A normal subgroup $K$ of a  group $G$, $\mathfrak{F}=BLF(f)$.}
Compute a chief series $K=G_0\trianglelefteq G_1\trianglelefteq\dots\trianglelefteq G_m=G$ of $G$\;
$F\gets K$\;
\For{$i\in\{1,\dots, m\}$}
  {\If{$(G_i/G_{i-1})'_{f(G_i/G_{i-1})}=(G_i/G_{i-1})'$}
          {$J\gets$type of $H/K$\;
           $F_1/F\gets \mathrm{O}_{\{J\}}(G_i/F)$\;
           $F\gets F_1$\;}
  }
\Return{$F/K$}
\end{algorithm}

\subsection{Proof of the theorem}

Let $K=G_0\trianglelefteq G_1\trianglelefteq\dots\trianglelefteq G_k=G$ be a part of chief series of $G$ (it can be computed in polynomial time by \ref{i5} of Theorem \ref{Basic}). Now by Lemma \ref{center}
 $$ I/K=\bigcap_{i=1, f((G_i/K)/(G_{i-1}/K))\neq \emptyset}^m C_{G/K, f}((G_i/K)/(G_{i-1}/K))
 =\bigcap_{i=1, f(G_i/G_{i-1})\neq \emptyset}^m C_{G, f}(G_i/G_{i-1})/K.$$
Hence this subgroup can be computed in a polynomial time by \ref{i7} of Theorem \ref{Basic} and Lemma \ref{fcent}. Using Lemma \ref{rad} we can compute
$R/K=  (I/K)_{E\mathfrak{F}}$. Now $R/K=(G/K)_\mathfrak{F}$ by Thoerem \ref{int}.

\begin{algorithm}[H]
\caption{FRADICAL$(G, K, \mathfrak{F})$}
\SetAlgoLined
\KwResult{$(G/K)_\mathfrak{F}$.}
\KwData{A normal subgroup $K$ of a  group $G$, $\mathfrak{F}=BLF(f)$.}

$T\gets G$\;
Compute a chief series $K=G_0\trianglelefteq G_1\trianglelefteq\dots\trianglelefteq G_m=G$ of $G$\;

\For{$i\in\{1,\dots, m\}$}
  {\If{$f(G_i/G_{i-1})\neq\emptyset$}
          {$C/G_{i-1}\gets C_{G/G_{i-1}}(G_i/G_{i-1})$\;
           $F/C\gets (G/C)_{f(G_i/G_{i-1})}$\; 
          $T/K\gets T/K\cap F/K$\;}
  }

\Return{EFRADICAL$(T, K, \mathfrak{F})$}

\end{algorithm}

\section{Applications}

\subsection{The $\mathfrak{F}$-Radical for a Primitive Saturated Formation}

From the fundamental result of Bryce and Cossey \cite{Bryce1982} the hereditary Fitting class of soluble groups is a primitive saturated formation.
Let $\mathcal{F}_0$ denote the family consisting of the empty set, the formation of groups of
order one, and the formation of all soluble groups, and then, for $i>0$,
define $\mathcal{F}_i$   inductively by $\mathfrak{F}\in\mathcal{F}_i$, if either $\mathfrak{F}\in\mathcal{F}_{i-1}$, or $\mathfrak{F}$ is a local formation,
with local definition $f$ such that $f(p)\in\mathcal{F}_{i-1}$, for all prime $p$. Finally
let $\mathcal{F}$  be the family comprising all formations $\mathfrak{F}$ such that $\mathfrak{F}=\cup_j \mathfrak{F}_j$ with each $\mathfrak{F}_j\in\cup_i\mathcal{F}_i$ and $\mathfrak{F}_j\subseteq \mathfrak{F}_{j+1}$. Formations from $\mathcal{F}$ are called primitive \cite[VII, Definition 3.1]{Doerk1992}. As was mentioned in \cite{Bryce1972} if $G\in\mathfrak{F}\in\mathcal{F}$ and the nilpotent length of $G$ is less than $m$, then there exists $\mathfrak{H}\in\mathcal{F}_m$ with $G\in\mathfrak{H}$.


\begin{thm}
  Let $m$ be a natural number. If $\mathfrak{F}\in\mathcal{F}_m$   and $K\leq G\leq S_n$, then $(G/K)_\mathfrak{F}$ can be computed in a polynomial time $($in $n)$.
\end{thm}

\begin{proof}
  It is clear that if $i=0$, then $(G/K)_\mathfrak{F}$ can be computed in polynomial time by \ref{i4} of Theorem \ref{Basic}. Assume that we can compute $(G/K)_\mathfrak{F}$ in polynomial time for every $\mathfrak{F}\in\mathcal{F}_{i-1}$. If $\mathfrak{F}\in\mathcal{F}_{i}\setminus\mathcal{F}_{i-1}$, then the values $f(p)$ of local definition of $\mathfrak{F}$ are in $\mathcal{F}_{i-1}$. By our assumption we can compute  $f(p)$-radicals in polynomial time. Hence we can compute $\mathfrak{F}$-radical in polynomial time in $n$ by Theorem \ref{thm2}.
\end{proof}

\subsection{The $\mathfrak{F}$-Length}

With some modifications the generalized Fitting subgroup can be computed in polynomial time by Theorem \ref{thm2}.

\begin{thm}\label{star}
  For $K\trianglelefteq G\leq S_n$ the generalized Fitting subgroup $\mathrm{F}^*(G/K)$ of $G/K$ can be computed in polynomial time.
\end{thm}

\begin{proof}
Recall that the generalized Fitting subgroup is an $\mathfrak{N}^*$-radical for a class $\mathfrak{N}^*$ of all quasinilpotent groups. This is class is a Baer-local formation defined by $h$ where $h(J)=1$ if $J$ is abelian and $h(J)=D_0(J)$ otherwise \cite[IX, Lemma 2.6]{Doerk1992}. Hence $C_{G, h}(H/K)=C_G(H/K)=HC_G(H/K)$ if $H/K$ is abelian. Note that if $H/K $ is non-abelian (of type $J$), then $G/C_G(H/K)$ has the unique minimal normal subgroup which is isomorphic to $H/K$ and hence coincides with $h(J)$-radical of $G/C_G(H/K)$. Thus $C_{G, h}(H/K)=HC_G(H/K)$  in this case. Note that $HC_G(H/K)$ can be computed in polynomial time by Theorem \ref{Basic}. Hence from the proof of Theorem \ref{thm2} it follows that $\mathrm{F}^*(G/K)$ of $G/K$ can be computed in polynomial time.\end{proof}

Recall that the nilpotent length $h(G)$ of a finite group $G$ is the least number $h$ such that $\mathrm{F}_h(G) = G$, where $\mathrm{F}_{(0)}(G) = 1$, and $\mathrm{F}_{(i+1)}(G)$ is the inverse image of $\mathrm{F}(G/\mathrm{F}_{(i)} (G))$. Note that the nilpotent length is defined only for soluble groups.
The $p$-length $l_p(G)$ of a finite group can be defined by $l_p(G)=0$ if $p\not\in\pi(G)$ and $l_p(G)=l_p(G/\mathrm{O}_{p',p}(G))+1$ otherwise. Note that the $p$-length is defined only for $p$-soluble groups.

\begin{defn}[{Khukhro, Shumyatsky \cite{Khukhro2015a, Khukhro2015b}}]
  $(1)$ The generalized Fitting height $h^*(G)$ of a finite group $G$ is the least number $h$ such that $\mathrm{F}_h^* (G) = G$, where $\mathrm{F}_{(0)}^* (G) = 1$, and $\mathrm{F}_{(i+1)}^*(G)$ is the inverse image of the generalized Fitting subgroup $\mathrm{F}^*(G/\mathrm{F}^*_{(i)} (G))$.

  $(2)$  Let $p$ be a prime, $1=G_0\leq G_1\leq\dots\leq G_{2h+1}=G$ be the shortest normal series in which for $i$ odd the factor $G_{i+1}/G_i$ is $p$-soluble $($possibly trivial$)$,
and for $i$ even the factor $G_{i+1}/G_i$ is a $($non-empty$)$ direct product of nonabelian simple
groups. Then $h=\lambda_p(G)$ is called the non-$p$-soluble length of a group $G$.

$(3)$  $\lambda_2(G)=\lambda(G)$ is the nonsoluble length of a group $G$.
\end{defn}

We use $R_p(G)$ to denote the $p$-soluble radical of a group $G$.  Let ${\mathrm{\overline{F}}^*_p}(G)$ be  the inverse image of $\mathrm{F}^*(G/R_p(G))$. Note that from \cite[Lemma 2.7]{Murashka2023} and $R_p(G/R_p(G))\simeq 1$ it follows that
 $\lambda_p(G)$   be defined by $\lambda_p(G)=0$ if $G$ is $p$-soluble and $
 \lambda_p(G)=\lambda_p(G/{\mathrm{\overline{F}}^*_p}(G))+1$ otherwise.

\begin{thm}
  If   $K\leq G\leq S_n$ and $p$ is a prime, then $h(G/K)$, $l_p(G/K)$, $h^*(G/K)$, $\lambda_p(G/K)$, $\lambda(G/K)$ can be computed in a polynomial time.
\end{thm}


\begin{proof}
 Assume that the $\mathfrak{F}$-radical exists in every group from a given homomorph $\mathfrak{H}$ and $(G/K)_\mathfrak{F}$ can be computed in a polynomial time for every $K\trianglelefteq G\leq S_n$ with $G/K\in\mathfrak{H}$. Moreover assume that Condition(1) and Condition(2) when $G/K\not\in\mathfrak{H}$ and $\mathfrak{F}$-length of $G/K$ is equal to 0 respectively can be checked in polynomial time. Then  the $\mathfrak{F}$-length can be computed in polynomial time  by the  algorithm ``FLENGTH''.

Note that $R_p(G/K)$ can be computed in polynomial time by \ref{i4} of Theorem \ref{Basic}. Therefore $\mathrm{F}^*(G/K)$ and ${\mathrm{\overline{F}}^*_p}(G/K)$ can be computed in polynomial time by Theorem \ref{star}.

Recall that the classes of all nilpotent and $p$-nilpotent groups can be locally defined by $f_1$ and $f_2$ respectively where $f_1(q)=1$ for all prime  $q$ and $f_2(p)=1$ and $f_2(q)=\mathfrak{G}$ for all prime $q\neq p$. Hence $\mathrm{F}(G/K)$ and $\mathrm{O}_{p',p}(G/K)$ can be computed in polynomial time by Theorem \ref{thm2}.

For $h$ Condition(1) is just ``is not soluble'', for  $l_p $  Condition(1) is just ``is not $p$-soluble'', for $h^*$ and $ \lambda_p$ we don't need to check Condition(1).
For $h$ and $h^*$ Condition(2) becomes ``is the unit group'', for $l_p$ it becomes ``is a  $p$-group'', for $\lambda_p$ it becomes ``is a  $p$-soluble group''. With the help of Theorem \ref{Basic} we can check this conditions in polynomial time.
 \end{proof}

\begin{algorithm}[H]
\caption{FLENGTH$(G, K, \mathfrak{F})$}
\SetAlgoLined
\KwResult{$\mathfrak{F}$ of $G/K$.}
\KwData{A normal subgroup $K$ of a  group $G$, $\mathfrak{F}$.}

\If{Condition(1)}{\Return{$\infty$}\;}
\eIf{Condition(2)}{\Return{0}\;}
{$T/K\gets (G/K)_\mathfrak{F}$\;
  \Return{FLENGTH$(G, T, \mathfrak{F})$+1}\;}
\end{algorithm}

\subsection*{Acknowledgments}

This work is supported by BFFR $\Phi23\textrm{PH}\Phi\textrm{-}237$.\\
I am grateful to A.\,F. Vasil'ev for helpful discussions.

{\small\bibliographystyle{siam}
\bibliography{Alg3}}

\end{document}